\documentclass{article}
\usepackage[english]{babel}
\usepackage[utf8]{inputenc}
\usepackage[T1]{fontenc}
\usepackage{graphicx}
\usepackage{geometry}
\usepackage{sectsty}
\usepackage{amsthm,amsfonts}
\usepackage{amsmath}
\usepackage{mathabx}
\usepackage{accents}

\sectionfont{\large}
\newtheorem{teo}{Theorem}
\newtheorem{lema}[teo]{Lemma}
\newtheorem{cor}[teo]{Corollary}
\newtheorem{prop}[teo]{Proposition}
\newtheorem{defi}[teo]{Definition}
\newtheoremstyle{mytheoremstyle} % name
    {\topsep}                    % Space above
    {\topsep}                    % Space below
    {}                   % Body font
    {}                           % Indent amount
    {\scshape}                   % Theorem head font
    {.}                          % Punctuation after theorem head
    {.5em}                       % Space after theorem head
    {}  % Theorem head spec (can be left empty, meaning ‘normal’)

\theoremstyle{mytheoremstyle} \newtheorem{nota}{Remark}
\theoremstyle{mytheoremstyle} \newtheorem{exemplo}{Example}
\numberwithin{equation}{section}
\newcommand{\real}{\mathbb{R}}
\newcommand{\complex}{\mathbb{C}}
\newcommand{\nat}{\mathbb{N}}
\newcommand \ben {\begin{equation}}
\newcommand \een {\end{equation}}
\newcommand \be {\begin{equation*}}
\newcommand \ee {\end{equation*}}
\newcommand \bi {\begin{itemize}}
\newcommand \ei {\end{itemize}}
\newcommand{\ubar}[1]{\underaccent{\bar}{#1}}

\date{}
\title{\textbf{Ground-states for systems of \textit{M} coupled semilinear Schrödinger equations with attraction-repulsion effects: characterization and perturbation results}}
\author{Simão Correia\\ \textit{CMAF-UL and FCUL, Av.\ Prof.\ Gama Pinto 2,}\\
\textit{1649-003 Lisboa, Portugal}\\
\textit{Email adress: sfcorreia@fc.ul.pt}}
\begin{document}
\maketitle

\begin{abstract}
We focus on the study of ground-states for the system of $M$ coupled semilinear Schrödinger equations with power-type nonlinearities and couplings. We extend the characterization result in \cite{simao} to the case where both attraction and repulsion are present and cannot be studied separately. Furthermore, we derive some perturbation and classification results to study the general system where components may be out of phase. In particular, we present several conditions to the existence of nontrivial ground-states. 
\end{abstract}

\begin{section}{Introduction}
In this work, we consider the system of $M$ coupled semilinear Schrödinger equations
\ben
i(v_i)_t + \Delta v_i + \sum_{j=1}^M k_{ij}|v_j|^{p+1}|v_i|^{p-1}v_i=0,\quad i=1,...,M
\een
where $V=(v_1,...,v_M):\real^+\times\Omega\to\real^M$, $\Omega\subset\real^N$ open with smooth boundary, $k_{ij}\in\real$, $k_{ij}=k_{ji}$, and $0<p<2/(N-2)^+$ (we use the convention $2/(N-2)^+=+\infty$, if $N=1,2$, and $2/(N-2)^+=2/(N-2)$, if $N\ge 3$). Given $1\le i\neq j\le M$, if $k_{ij}\ge 0$, one says that the coupling between the components $v_i$ and $v_j$ is attractive; if $k_{ij}< 0$, it is repulsive. 

When we look for nontrivial periodic solutions of the form $V=e^{i\omega t}U$, with $U=(u_1,...,u_M)\in (H_0^1(\Omega))^M$ (called bound-states), we are led to the study of the system
\begin{equation}\tag{M-NLS}\label{BS}
 \Delta u_i - \omega u_i + \sum_{j=1}^M k_{ij}|u_j|^{p+1}|u_i|^{p-1}u_i=0 \quad i=1,...,M.
\end{equation}

On the other hand, one may also consider periodic solutions where the time-frequency is not necessarily the same for each component (one then says that the components are out of phase). These solutions are of the form $V=(e^{i\omega_1t}u_1,..., e^{i\omega_Mt}u_M)$ and the stationary system is
\begin{equation}\tag{M-NLS'}\label{BSomega}
 \Delta u_i - \omega_i u_i + \sum_{j=1}^M k_{ij}|u_j|^{p+1}|u_i|^{p-1}u_i=0 \quad i=1,...,M.
\end{equation}

Notice that, if $M=1$ and $\Omega=\real^N$, the presence of $\omega>0$ may be eliminated by a suitable scaling. However, in any other case, such a procedure is no longer possible.

In any case, for both physical and mathematical reasons, one is interested in bound-states which have minimal action among all bound-states, the so-called ground-states. The set of such solutions is noted $G$. For $\Omega=\real^N$, in the scalar case, one may prove that there is a unique ground-state $Q$ (modulo translations and rotations, see \cite{cazenave}). For a general $\Omega$, the problem has not been completely solved. However, it is known, for example, that there exists a ground-state if
\ben\label{deflambda1}
\omega>-\lambda_1(\Omega),\  \lambda_1(\Omega)=\left\{\begin{array}{cc}
\mbox{ first eigenvalue of } -\Delta \mbox{ on } H^1_0(\Omega),& \Omega \mbox{ bounded}\\ 0 & \Omega \mbox{ infinite parallelipiped}

\end{array}\right..
\een

The vector case is much more complex. The existence of ground-states for system \eqref{BSomega} on $\Omega=\real^N$ has been proven under the sufficient and necessary condition
\ben\label{condicaoexistencia}
\exists U=(u_1,...,u_M)\in (H^1(\real^N))^M:\ \sum_{i,j=1} k_{ij}\int |u_i|^{p+1}|u_j|^{p+1} >0
\een
using a suitable variational formulation. We note that the result is still true for any $\Omega$. In fact, one proves that the set of ground-states is the set of minimizers of
\ben
\inf\left\{\int \sum_{i=1}^M \omega_i|u_i|^2 + |\nabla u_i|^2: \sum_{i,j=1} k_{ij}\int |u_i|^{p+1}|u_j|^{p+1} =\lambda \right\},
\een
for a precise and explicit $\lambda$. To prove existence of minimizers, the main difficulty is the strong compactness of the minimizing sequence in $L^{2p+2}$. In $\Omega$ bounded, this is trivial, since one has the compact injection $H_0^1(\Omega)\hookrightarrow L^{2p+2}(\Omega)$. For $\Omega=\real^N$, one uses the concentration-compactness principle and proves the compactness alternative. For $\Omega$ an infinite parallelipiped, one simply extends the minimizing sequence to $\real^N$ by $0$ and apply the technique for the whole space. Since the existence of ground-states for general $\Omega$ is an open problem, we shall make the following assumption
\begin{equation}\label{exist}\tag{Exist}
\mbox{"The set of all ground-states for (M-NLS) over } \Omega,\ G \mbox{, is nonempty."}
\end{equation}

One then may pose a number of questions: is there a unique positive ground-state? Does a ground-state have all components different from $0$ (called nontrivial ground-states)? Can we obtain a simple characterization of the family of ground-states? Are the solutions positive and radially decreasing?

Regarding system \eqref{BS}, for $\Omega=\real^N$, a recent work (\cite{simao}) has answered to these questions for a very large family of matrices $K=(k_{ij})_{1\le i, j\le M}$. Essentially, if one may group the components in such a way that two components attract each other if and only if they belong to the same group, one may answer all questions above in a satisfactory fashion. One may also prove that, in the case where all components attract each other, the result is extendible to any $\Omega$. If this grouping hypothesis fails, the situation becomes much more difficult. The reason is that two components may repel each other directly but, by transitivity, they also attract each other (see figure 1). Then the balance between these forces is not clear and the analysis is not straightfoward.
\begin{figure}[h]
\centering
\includegraphics[width=7cm, height=3cm]{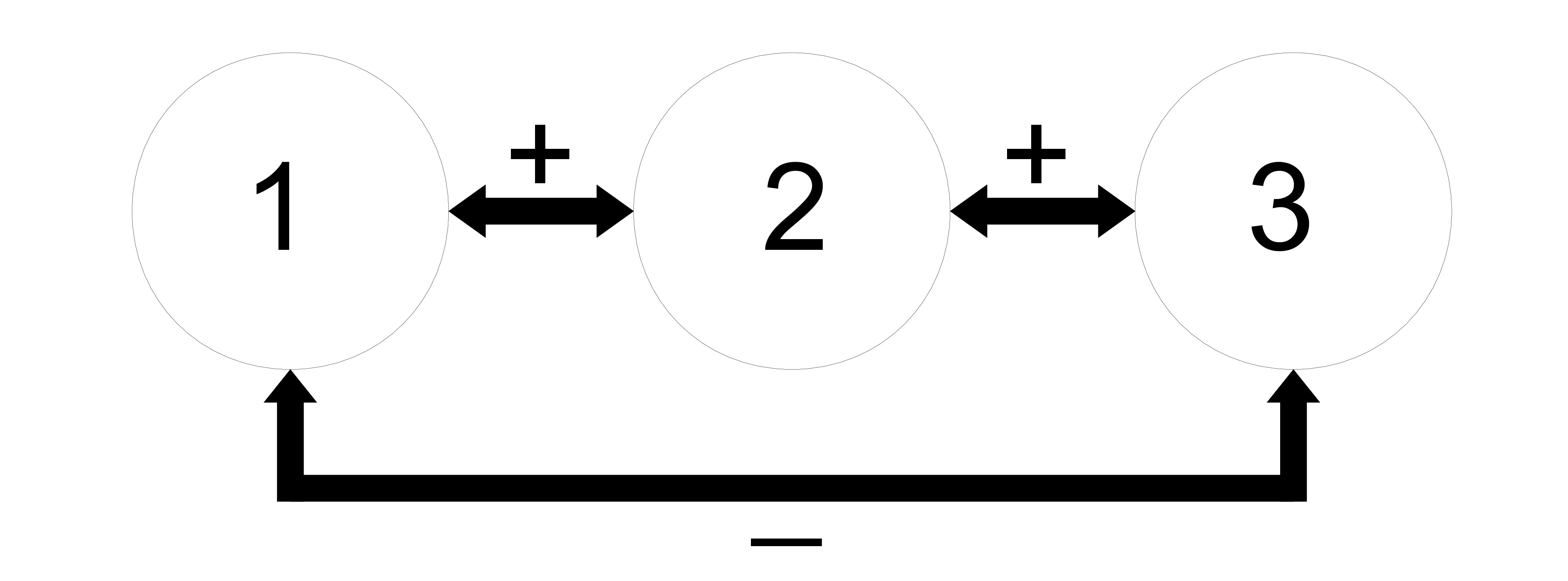}
\caption{The simplest balanced system: the signs indicate wether the components attract or repel each other. Though components $1$ and $3$ repel each other, they are both attracted to component $2$. This case was studied for the first time in \cite{linwei}.}
\end{figure}

For the general system \eqref{BSomega}, it is impossible in general to obtain a characterization similar to the one for \eqref{BS}. Results obtained so far consider mostly the case $M=2$, $\Omega=\real^N$. For this case, the behaviour of the system with respect to the parameters is very well understood. Uniqueness of positive radial solution has been considered in \cite{mazhao} and \cite{weiyao}, mostly through a careful analysis of the system of ODE's that one obtains when considering radial solutions. The question that has attracted more attention in the past few years is the existence of nontrivial ground-states. We advise the reader to check \cite{maiamontefuscopellacci}, \cite{liuliuchang}, \cite{mandel} (and references therein).

In this work, we shall consider both systems \eqref{BS} and \eqref{BSomega}. For system (M-NLS), we complete the work started in \cite{simao} and obtain the characterization of ground-states regardless of coupling coefficients. We note by $G^+$ the set of nontrivial ground-states.
\begin{teo}\label{caract}
Consider system (M-NLS) and suppose \eqref{exist}. Define $f:(\real_0^+)^M\to \real$,
\ben
f(X)=\sum_{i,j=1}^M k_{ij}x_i^{p+1}x_j^{p+1}
\een
and let $\mathcal{X}\subset (\real_0^+)^M$ be the set of solutions of 
\ben
f(X_0)=f_{max}:=\max_{|X|=1}f(X),\quad |X_0|=1.
\een
Then $U\in G$ if and only if there exist $a_i\in\complex, 1\le i\le M$, such that
$(f_{max})^{1/2p}(|a_1|,...,|a_M|)\in\mathcal{X}$ and $u_0$ ground-state of
\ben\label{equ_0}
\Delta u - \omega u +|u|^{2p}u=0\mbox{ on } \Omega
\een
such that
\ben\label{formula}
U=(a_iu_0)_{1\le i\le M}.
\een
In particular, $G^+\neq\emptyset$ if and only if there exists $X\in \mathcal{X}$ such that $X_i\neq 0,\ i=1,...,M$. Moreover $G=G^+$ if and only if all elements of $\mathcal{X}$ have no zero components.
\end{teo}

\begin{nota}
The fact that the constants appearing in \eqref{formula} do not depend on $\Omega$ is a remarkable property. As a consequence, the question of wether $G^+$ is empty or not is also independent on $\Omega$. For example, we know that, for $M=3$, $k_{12}>0$, $k_{13},k_{23}<0$ and $\Omega=\real^N$, either $a_1=a_2=0$ or $a_3=0$ (see \cite{simao}). This has been proven by arguing that translating the third component of a ground-state to the infinite decreases the action, which is not an available argument for $\Omega$ bounded. Now, however, we see that the result is also true for any $\Omega$ for which \eqref{exist} holds, in particular over bounded domains.
\end{nota}

\begin{nota}
In \cite{linwei}, it is considered the case $\Omega=\real^N$, $M=3$, $p=1$, $k_{12},k_{23}>0$ and $k_{13}<0$. They prove that if $k_{ii}=1$, $i=1,2,3$, $k_{12},k_{23} \approx \delta^2$ and $k_{13}\approx -\delta$, $\delta>0$ small, any nontrivial ground-state is not radial. This implies that such a ground-state cannot be of the form $U=(a_iQ)_{1\le i\le 3}$, where $Q$ is the unique ground-state for the scalar equation, since $Q$ is radial. Together with the above theorem, one sees that there are no nontrivial ground-states for this system.
We claim that it is possible to obtain such a conclusion in a more general setting. In \cite{simao}, it is proved that, for $p\ge 1$ and $k_{ij}>0, \forall i, j$, there exists $\epsilon>0$ such that, if $\max_{i\neq j} |k_{ij}|<\epsilon$, there are no nontrivial ground-states. To prove this, one uses the implicit function theorem to determine the constants of the characterization formula as a perturbation of the system where $k_{ij}=0, i\neq j$. Afterwards, the computation of the action proves that the ground-state is semitrivial. We now notice that this proof still works without the restriction $k_{ij}>0$. In fact, this restriction was made only because the characterization result available needed such an hypothesis.
\end{nota}

For \eqref{BSomega}, since a reduction to the scalar case is impossible, the main questions are about existence of nontrivial ground-states (one might also discuss uniqueness, but that is a difficult matter even for the (M-NLS) system, where we have a complete characterization). Our results focus on two approaches: the first considers perturbation of the parameters of the system, while the second considers a real-valued function on the parameters whose properties determine the emptiness of $G^+$. We now explain the main ideas.

\begin{center}
\textbf{Approach 1: Perturbation theory}
\end{center}

First of all, a scaling reduces any (M-NLS') system to the case $\omega\ge 1$. Given a nonempty symmetric subset $P$ of $\{1,...,M\}^2$, $\beta\in\real$ and $\eta>0$, consider, for $i=1,...,M$,
\ben
\Delta u_i - (1+\eta(\omega_i-1)) u_i + \sum_{(i,j)\notin P} k_{ij}|u_j|^{p+1}|u_i|^{p-1}u_i + \sum_{(i,j)\in P} \beta k_{ij}|u_j|^{p+1}|u_i|^{p-1}u_i =0
\een
For the sake of simplicity, suppose that $k_{ij}>0, \forall (i,j)\in P$. If one considers the ground-state action level, $\mathcal{I}_\beta^\eta$, and the semitrivial ground-state action level, $(\mathcal{I}_\beta^\eta)^{sem}$, then $\mathcal{I}_\beta^\eta<(\mathcal{I}_\beta^\eta)^{sem}$ is equivalent to $G=G^+$. The continuity of these action levels with respect to $\beta$ and $\eta$ leads to perturbation results: if, for some $\beta_0, \eta_0$, one proves that the ground-state action level is strictly lower than the semitrivial action level, then the same inequality is valid for $\beta,\eta$ close to $\beta_0,\eta_0$. We exemplify such an argument with two corollaries:
\begin{cor}\label{semitrivialinfinito}
Consider system (M-NLS').
\begin{enumerate}
\item If $M=2$ and $0<k_{11}, k_{22}\ll k_{12}$, $G=G^+$;
\item For $M\ge 3$, if $k_{ii}=-1, \ \forall i$ and $k_{ij}=\beta, \ \forall i\neq j$, there exists $\epsilon>0$ such that, if
\ben
\frac{2}{M-1}<\beta<\frac{2}{M-2} + \epsilon,
\een
then $G=G^+$.
\end{enumerate}
\end{cor}

\begin{cor}\label{diagonalpequena}
Consider system (M-NLS'), $M\ge3$. Suppose that $\Omega=\real^N$, $p\le 1$, $\omega_1\le\omega_2\le ...\le \omega_M$ and $k_{ij}=b>0,\ \forall i\neq j$. Assume that
\ben\label{hipoteseomegasdiferentes}
\frac{M-1}{(M-2)^{1/p}}>\frac{M}{(M-1)^{1/p}}\left(\frac{\omega_M}{\omega_1}\right)^{\frac{2-p(N-2)}{2p}}.
\een
Then there exists $\delta>0$ such that, if $\max_i |k_{ii}|< \delta b$, $G=G^+$.
\end{cor}

Note that the first part of corollary \ref{semitrivialinfinito} is already known (see \cite{mandel} and \cite{maiamontefuscopellacci}). Also, in corollary \ref{diagonalpequena}, if $\omega_1=\omega_M$, the result is a particular case of \cite{liuwang} and \cite{oliveiratavares}. Even so, we prove these results for two reasons: first, the proof is very simple when one looks from this pertubative perspective; second, the approach is rather different in nature and it deals only with continuity properties, which may have a greater capacity of generalization to other systems.

Regarding corollary \ref{diagonalpequena}, a comment is in need: it might be expected that the restriction $p\le 1$ would be technical. In fact, for $M=2$, the result is valid for any $p>0$. By contrast, we prove

\begin{prop}\label{p=3}
Consider system (3-NLS), with $p=3$, $k_{ii}=0, \forall i$ and $k_{ij}=1, \forall i\neq j$. Then $G^+=\emptyset$.
\end{prop}

We conjecture that the above result applies for more general $M$ and $p$, with $k_{ii}=\mu, \forall i,$ and $k_{ij}=b, \ \forall i\neq j$, $\mu\ll b$. In fact, a necessary and sufficient condition for the existence of nontrivial ground-states should be (see equation \eqref{Itodosiguais})
\ben
\frac{M}{(M-1)^{1/p}}\le \frac{M-1}{(M-2)^{1/p}}.
\een
In fact, \textit{if the only possible nontrivial ground-state is the one with all components equal}, this condition determines wether it truly is a ground-state. Numerical simulations suggest that this uniqueness should hold for any $p,M$. We advise the reader to compare this hypothesis with the condition for existence of nontrivial ground-states that appears in \cite{liuwang}.

\begin{center}
\textbf{Approach 2: Mandel's characteristic function}
\end{center}
Once again, given a nonempty symmetric subset $P$ of $\{1,...,M\}^2$ and $\beta\in\real$, consider the following system:

\ben
\Delta u_i - \omega_i u_i + \sum_{(i,j)\notin P} k_{ij}|u_j|^{p+1}|u_i|^{p-1}u_i + \sum_{(i,j)\in P} \beta k_{ij}|u_j|^{p+1}|u_i|^{p-1}u_i =0,\quad i=1,...,M.
\een

In section 5, we shall build a mapping $\beta\mapsto\hat{\beta}$ such that, if
\bi
\item $\beta<\hat{\beta}$, then $G_\beta^+=\emptyset$;
\item $\beta>\hat{\beta}$, then $G_\beta^+=G_\beta$; 
\item $\hat{\beta}=\beta$, then $G_\beta\setminus G_\beta^+\neq\emptyset$.
\ei

This approach was introduced by R. Mandel (\cite{mandel}) for the system with two equations to study the existence of nontrivial ground-states as a function of the coupling coefficient $k_{12}$. The (very important) feature of the case $M=2$ is that any semitrivial bound-state is never influenced by the coupling coefficient. This implies that $\hat{\beta}$ is constant and therefore it defines in a very precise way when does $G^+\neq\emptyset$. For more equations, $\hat{\beta}$ is not that well-behaved (for more details, see section 5 and the last example in section 6). Using this mapping, we may however prove two results:
\begin{prop}\label{propmandel1}
Let $U$ be a semitrivial bound-state for (M-NLS') and suppose that $P\subset \{1,...,M\}$ is such that
\ben
(i,j)\in P\Rightarrow U_iU_j\equiv 0.
\een
Then, for $\beta$ large, $U\notin G$.
\end{prop}

\begin{prop}\label{propmandel2}
Consider system (M-NLS') and fix $p\ge 1$. Suppose that $k_{ij}=\beta>0,\ \forall i\neq j$, and $k_{ii}=\mu>0, \forall i$. If $\beta\ll \mu$, any ground-state has exactly one nonzero component.
\end{prop}
\vskip20pt

This work is organized as follows: in section 2, we give a few definitions and fix some notations. In section 3, we focus on the results regarding system (M-NLS). In sections 4 and 5, we study system (M-NLS'), the first using perturbation theory, the second using Mandel's characteristic function. Finally, in section 6, we give three examples: one to see the application of theorem \ref{caract}; the rest to show the complexity of these systems for $M=3$. We recall that we shall always assume \eqref{exist}.

\end{section}
\begin{section}{Preliminaries}

\begin{defi}\textit{(Bound-states and ground-states of (M-NLS))}
\begin{enumerate}

\item We define bound-state of (M-NLS) as any element $(u_1,...,u_M)\in (H^1_0(\Omega))^M\setminus \{0\}$ solution of \eqref{BS}
and define $A_{\footnotesize{\mbox{(M-NLS)}}}$ to be the set of all bound-states of (M-NLS).
\item A nontrivial bound-state is a bound-state such that $u_i\neq0, \ \forall i$. The set of such bound-states is called $A_{\footnotesize{\mbox{(M-NLS)}}}^+$. On the other hand, a bound-state which is not nontrivial is called semitrivial.
\item Given $U=(u_1,...,u_M)\in (H^1_0(\Omega)))^M$, set
\ben
I_M(U)=\sum_{i=1}^M \int |\nabla u_i|^2 + \int \omega_i |u_i|^2, \ J_M(U)=\sum_{i,j=1}^M k_{ij}\int |u_i|^{p+1}|u_j|^{p+1}
\een
and define the action of $U$,
\ben
S_M(U)=\frac{1}{2}I_M(U) - \frac{1}{2p+2}J_M(U).
\een
\item The set of ground-states of (M-NLS) is defined as
\ben
G_{\footnotesize{\mbox{(M-NLS)}}}=\{U\in A_{\footnotesize{\mbox{(M-NLS)}}}: S_M(U)\le S_M(W),\ \forall W\in A_{\footnotesize{\mbox{(M-NLS)}}}\}\subset A_{\footnotesize{\mbox{(M-NLS)}}},
\een
and the set of nontrivial ground-states is
\ben
G_{\footnotesize{\mbox{(M-NLS)}}}^+=G_{\footnotesize{\mbox{(M-NLS)}}}\cap A_{\footnotesize{\mbox{(M-NLS)}}}^+.
\een
\end{enumerate}
\end{defi}

\begin{nota}
If $U\in A_{\footnotesize{\mbox{\textit{(M-NLS)}}}}$, $I_M(U)=J_M(U)$ (one multiplies the $i$-th equation by $u_i$ and integrates over $\real^N$). Therefore
\ben
S_M(U)=\left(\frac{1}{2}-\frac{1}{2p+2}\right)I_M(U)=\left(\frac{1}{2}-\frac{1}{2p+2}\right)J_M(U).
\een
Hence a ground-state is a bound-state with $I_M$ (or $J_M$) minimal.
\end{nota}
\begin{nota}
Throughout this work, we shall assume that $k_{ij}$ are such that
\be\tag{P1}
\{U\in (H^1_0(\Omega))^M: J_M(U)>0\}\neq\emptyset.
\ee
This hypothesis is necessary for the existence of bound-states, since $J_M(U)=I_M(U)>0$, for any $U\in A_{\footnotesize{\mbox{\textit{(M-NLS)}}}}$. Furthermore, we shall assume that $\omega>0$.
\end{nota}
\begin{nota}
Since $M\ge 2$ will always be fixed, to simplify notations, we write
\ben
A:=A_{\footnotesize{\mbox{\textit{(M-NLS)}}}},\ G:=G_{\footnotesize{\textit{\mbox{(M-NLS)}}}},\ G^+:=G^+_{\footnotesize{\mbox{\textit{(M-NLS)}}}}
\een
and
\ben
I:=I_M,\ J:=J_M,\ S:=S_M.
\een
\end{nota}

%The following two lemmas will be of importance (see [CAZ] and [SIMAO]).
%
%\begin{lema}\label{unicidadeQ}
%There exists $Q\in H^1(\real^N)\setminus\{0\}$ radial, positive and strictly decreasing such that
%\ben
%G_{\footnotesize{\mbox{(1-NLS)}}}=\{e^{i\theta}Q(\cdot + y): \theta\in\real,\ y\in\real^N\}.
%\een
%\end{lema}

Let
\ben
\lambda_G:=\left(\inf_{J(U)=1} I(U)\right)^{\frac{p+1}{p}}.
\een

The following lemma, which may be found in \cite{simao}, gives a variational characterization of the set of ground-states:

\begin{lema}
Under hypothesis (P1) and (Exist), $G$ is the set of solutions of the minimization problem
\ben\label{minlambdaG}
I(U)=\min_{J(W)=\lambda_G} I(W),\quad J(U)=\lambda_G.
\een
Moreover, if $\Omega=\real^N$, (Exist) holds.
\end{lema}

\end{section}
\begin{section}{Proof of theorem 1}

Take $U\in G$. Define $\hat{U}(x)=(|u_1(x)|,..., |u_M(x)|)$ and $u(x)=|\hat{U}(x)|$. Since $J(U)=J(\hat{U})$ and $I(U)\ge I(\hat{U})$, $\hat{U}$ is a minimizer. Fix $X_0\in\mathcal{X}$. Now notice that
\begin{align*}
J(\hat{U})&=\int f(\hat{U}(x)) dx = \int f\left(\frac{\hat{U}(x)}{u(x)}\right)u(x)^{2p+2} dx \le \int f(X_0)u(x)^{2p+2}dx \\&= \int f\left(X_0u(x)\right) dx = J(X_0u)
\end{align*}
and that, from Cauchy-Schwarz inequality,
\begin{align*}
I(X_0u)&=\int \omega u(x)^2|X_0|^2 + \left|\nabla (u(x))\right|^2|X_0|^2 = \int \sum_{i=1}^M\omega|u_i|^2 + \left|\frac{\sum_{i=1}^M |u_i|\nabla |u_i|}{\left(\sum_{i=1}^M |u_i|^2\right)^{\frac{1}{2}}}\right|^2 \\&\le \int \sum_{i=1}^M\omega|u_i|^2 +  |\nabla |u_i||^2=I(\hat{U}).
\end{align*}
Let $a\le 1$ be such that $J(aX_0u)=J(\hat{U})$. Then
\ben
I(aX_0u)\le I(X_0u) \le I(\hat{U})
\een By the minimality of $U$, the above inequalities must be equalities: 
\ben\label{U=|U|}
a=1,\quad I(X_0u)=I(U).
\een
Therefore $X_0u$ is also a ground-state. Note that $J(U)=J(X_0u)$ implies that $\hat{U}(x)=u(x)X(x)$ a.e. $x\in\real^N$, where $X(x)\in\mathcal{X}$.

\noindent Since $X_0u$ is a bound-state for (M-NLS), one easily checks that
\ben
-\Delta u + \omega u = f_{max}|u|^{2p}u
\een
and so, setting $c=(f_{max})^{1/2p}$, $cu$ is a bound-state for \eqref{equ_0}. The fact that $X_0u$ is a ground-state clearly implies that $cu$ is a ground-state for \eqref{equ_0}. Hence $u=c^{-1}u_0$, with $u_0$ ground-state of \eqref{equ_0}. From the maximum principle, $u>0$ in $\real^N$.

\noindent Since $\hat{U}(x)=u(x)X(x)$ is a bound-state, inserting this expression into system (M-NLS), one obtains
\ben
2\nabla u\cdot \nabla X_i + u\Delta X_i = 0, \quad i=1,...,M
\een
By integration by parts,
\begin{align*}
\int uX_i\nabla u\cdot \nabla X_i &= - \int uX_i\nabla u\cdot \nabla X_i - \int |u|^2|\nabla X_i|^2 -\int |u|^2X_i\Delta X_i \\&= - \int uX_i\nabla u\cdot \nabla X_i - \int |u|^2|\nabla X_i|^2 + 2\int  uX_i\nabla u\cdot \nabla X_i \\&= \int  uX_i\nabla u\cdot \nabla X_i - \int |u|^2|\nabla X_i|^2
\end{align*}
Hence
\ben
\int |u|^2|\nabla X_i|^2=0,\quad i=1,...,M.
\een
which implies that $X_i$ is constant. Therefore
\ben
\hat{U}=uX, X\in\mathcal{X}. 
\een
\noindent Finally, since $u>0$, one may write $u_i(x)=|u_i(x)|e^{i\theta(x)}=u(x)Xe^{i\theta(x)}$. Then, since $I(U)=I(\hat{U})$,
\begin{align*}
\int \sum_{i=1}^M \omega |u_i|^2+|\nabla |u_i||^2&= \int \sum_{i=1}^M \omega|\hat{U}_i|^2+|\nabla \hat{U}_i|^2 =I(\hat{U}) = I(U) \\&= \int \sum_{i=1}^M \omega|u_i|^2 + |\nabla u_i|^2 = \int \sum_{i=1}^M \omega|u_i|^2 + |\nabla |u_i||^2 + |u_i|^2|\nabla \theta_i(x)|^2.
\end{align*}
One then concludes that $\theta_i$ is constant, which ends the proof.$\ \qedsymbol$

\begin{nota}
As expected, this approach is only possible since the norms inside functional $I$ are the same. This is not the case for the general system \eqref{BSomega}.
\end{nota}

\end{section}

\begin{section}{System (M-NLS'): Perturbation theory}

\begin{lema}[Monotonicity of the action with respect to $\omega$]\label{mon1}
Let $\omega=(\omega_1,...,\omega_M)$ and $\omega'=(\omega_1',...,\omega_M')$ be such that $\omega\ge\omega'$. Fix a matrix $K=(k_{ij})_{1\le i,j\le M}\in\real^{M\times M}$. Let $U^\omega$ be a ground-state of
\begin{equation}
 \Delta u_i - \omega_i u_i + \sum_{j=1}^M k_{ij}|u_j|^{p+1}|u_i|^{p-1}u_i=0 \quad i=1,...,M
\end{equation}
and $U^{\omega'}$ be a ground-state of
\begin{equation}
 \Delta u_i -  \omega_i' u_i + \sum_{j=1}^M k_{ij}|u_j|^{p+1}|u_i|^{p-1}u_i=0 \quad i=1,...,M.
\end{equation}
Then $J(U^\omega)\ge J(U^{\omega'})$.
\end{lema}
\begin{proof}
Simply recall that
\ben
J(U^\omega)=\left(\inf_{J(U)=1} \sum_{i=1}^M \omega_i\|u_i\|_2^2 + \|\nabla u_i\|_2^2\right)^{\frac{p+1}{p}} \ge \left(\inf_{J(U)=1} \sum_{i=1}^M \omega_i'\|u_i\|_2^2 + \|\nabla u_i\|_2^2\right)^{\frac{p+1}{p}}=J(U^{\omega'}).
\een
\end{proof}

Analogously, we may obtain the following:

\begin{lema}[Monotonicity of the action with respect to $K$]\label{mon2}
Fix $\omega\in(\real^+)^M$. Consider matrices $K=(k_{ij})_{1\le i,j\le M}\in\real^{M\times M}$ and $K'=(k'_{ij})_{1\le i,j\le M}\in\real^{M\times M}$ such that $K\ge K'$. Let $U^K$ be a ground-state of
\begin{equation}
 \Delta u_i - \omega_i u_i + \sum_{j=1}^M k_{ij}|u_j|^{p+1}|u_i|^{p-1}u_i=0 \quad i=1,...,M
\end{equation}
and $U^{K'}$ be a ground-state of
\begin{equation}
 \Delta u_i -  \omega_i u_i + \sum_{j=1}^M k'_{ij}|u_j|^{p+1}|u_i|^{p-1}u_i=0 \quad i=1,...,M.
\end{equation}
Then $I(U^{K'})\ge I(U^K)$.
\end{lema}

Suppose that one wishes to study $G^+$ in function of a given set of couplings. Let $P$ a nonempty symmetric subset of $\{1,...,M\}^2$ and fix a matrix $K\in \real^{M^2}$. 
% We assume that $k_{ij}\ge 0, \forall (i,j)\in P$ (the case $k_{ij}\le 0, \forall (i,j)\in P$ may also be treated if one considers $-K$ instead of $K$). 
%Define
%\ben
%J_P(U)=\sum_{(i,j)\in P} k_{ij}\int |u_i|^{p+1}|u_j|^{p+1},\quad J_{NP}(U)=\sum_{(i,j)\notin P} k_{ij}\int |u_i|^{p+1}|u_j|^{p+1}
%\een
Given $\beta\in\real$, consider the system
\ben\label{MNLSbeta}
\Delta u_i - \omega_i u_i + \sum_{(i,j)\notin P} k_{ij}|u_j|^{p+1}|u_i|^{p-1}u_i + \sum_{(i,j)\in P} \beta k_{ij}|u_j|^{p+1}|u_i|^{p-1}u_i =0,\quad i=1,...,M
\een
Suppose, for the sake of simplicity, that $k_{ij}> 0, (i,j)\in P$. Everytime a functional, a set or a solution depends on $\beta$, we shall place a subscript $\beta$.

Set
\ben
\mathcal{I}_\beta=\left(\inf_{J_\beta(U)=1} I(U)\right)^{\frac{p+1}{p}}.
\een

For any $X\subset \{1,...,M\}$, define
\ben
\mathcal{I}_\beta^X:=\left(\inf_{J_\beta(U)=1, U_i=0, i\notin X} I(U)\right)^{\frac{p+1}{p}}, \quad \mathcal{I}_\beta^{sem}:=\min_{X\subsetneq \{1,...,M\}} \mathcal{I}_\beta^X.
\een
Notice that $\mathcal{I}_\beta=\mathcal{I}_\beta^{\{1,...,M\}}$. Then $G^+=G$ iff $\mathcal{I}_\beta< \mathcal{I}_\beta^{sem}$.

From the results regarding existence of ground-states, we know that, for each $X\subset \{1,...,M\}$, there exists $\ubar{\beta}_X$ such that $\beta\le \ubar{\beta}_X$ iff $\mathcal{I}_\beta^X=+\infty$. Define 
\ben
\ubar{\beta}^{sem}:=\min_{X\subsetneq \{1,...,M\}} \ubar{\beta}_X,\quad \ubar{\beta}:=\ubar{\beta}_{\{1,...,M\}}.
\een
Then
\begin{enumerate}
\item If $\beta\le \ubar{\beta}$, there are no ground-states;
\item If $\ubar{\beta}< \beta \le \ubar{\beta}^{sem}$, all ground-states are nontrivial;
\item If $\ubar{\beta}^{sem}<\beta$, both $\mathcal{I}_\beta$ and $\mathcal{I}_\beta^{sem}$ are finite.
\end{enumerate}

\begin{prop}
For any $X\subset \{1,...,M\}$, the mapping $\beta\mapsto \mathcal{I}_{\beta}^X, \beta\in\real$, is continuous (in $\overline{\real}$). In particular, $\mathcal{I}_\beta$ and $\mathcal{I}_\beta^{sem}$ are continuous with respect to $\beta$.
\end{prop} 
\begin{proof}
Notice that we only need to prove the proposition for $X=\{1,...,M\}$, since any other case may be reduced to this one.

\noindent Fix $\beta_0\in\real$. If $\beta_0<\ubar{\beta}$, then $\mathcal{I}_{\beta}\equiv +\infty$ in a neighbourhood of $\beta_0$ and so it is continuous.

\noindent If $\beta_0>\ubar{\beta}$, let $\beta_n\to\beta_0$. By definition, there exists $\{U_n\}\subset (H^1(\real^N))^M$ such that
\ben
I(U_n)=\mathcal{I}_{\beta_n},\quad J_{\beta_n}(U_n)=1.
\een
Let $\lambda_n=J_{\beta_0}(U_n)^{-1/2p}$. Then $J_{\beta_0}(\lambda_nU_n)=1$. Moreover,
\ben
|\lambda_n^{-1/2p} - 1|=|J_{\beta_0}(U_n) - J_{\beta_n}(U_n)|=|\beta_n-\beta_0||J_P(U_n)|
\een
Since 
\ben
|J_P(U_n)|\le C\|U_n\|_{H^1}^{2p+2}\le CI(U_n)^{2p+2} \le C(\mathcal{I}_{\beta_n})^{2p+2}<C,
\een
we obtain $\lambda_n\to 1$. Therefore,
\ben
\liminf \mathcal{I}_{\beta_n} = \liminf I(U_{\beta_n}) = \liminf I(\lambda_nU_{\beta_n}) \ge \mathcal{I}_{\beta_0}.
\een
On the other hand, for $n>0$, let $U$ be such that
\ben
\mathcal{I}_{\beta_0} = I(U),\quad J_{\beta_0}(U)=1.
\een
Define $\lambda^n=J_{\beta_n}(U)^{-1/2p}$. As before, $J_{\beta_n}(\lambda^nU)=1$ and $\lambda^n\to 1$. Hence
\ben
\mathcal{I}_{\beta_0}=  I(U) = \lim I(\lambda^nU) \ge \limsup \mathcal{I}_{\beta_n}.
\een
Therefore $\mathcal{I}_\beta$ is continuous for $\beta>\ubar{\beta}$.

\noindent If $\beta_0=\ubar{\beta}$ and $\beta_n\to \beta_0^+$, consider $U_n$ as above. Then
\ben
1=J_{\beta_n}(U) = J_{\beta_0}(U_n) + (\beta_n - \beta_0)J_P(U_n) \le C(\beta_n - \beta_0)I(U_n)^{2p+2}.
\een
and so $\mathcal{I}_{\beta_n}=I(U_n)\to\infty=\mathcal{I}_{\beta_0}$. 

\end{proof}

\begin{lema}
Suppose that $\ubar{\beta}<\ubar{\beta}^{sem}$. For $\beta$ sufficiently close to $\ubar{\beta}^{sem}$, $G=G^+$.
\end{lema}
\begin{proof}
Since $\ubar{\beta}<\ubar{\beta}^{sem}$, there exists $U$ nontrivial such that, for some $m>0$ and for $\beta$ close to $\ubar{\beta}^{sem}$,
\ben
J_{\beta}(U)>m.
\een
This implies that
\ben
\mathcal{I}_\beta\le \left(\frac{1}{m^{\frac{1}{p+1}}}I(U)\right)^{\frac{p+1}{p}}.
\een
On the other hand, since $\mathcal{I}_\beta^{sem}$ is continuous, for $\beta$ sufficiently close to $\ubar{\beta}^{sem}$,
\ben
\mathcal{I}_\beta^{sem}> \left(\frac{1}{m^{\frac{1}{p+1}}}I(U)\right)^{\frac{p+1}{p}}\ge \mathcal{I}_\beta.
\een
Therefore $G=G^+$.
\end{proof}
\noindent\textit{Proof of corollary \ref{semitrivialinfinito}}. 
\begin{enumerate}
\item First part: take $P=\{(1,1),(2,2)\}$. One easily observes that $\ubar{\beta}^{sem}=0$ and that $\ubar{\beta}<0$. Therefore, using the previous lemma, for $\beta>0$ small enough, $G=G^+$.

\item Second part: take $P=\{(i,j), \ 1\le i,j\le M,\ i\neq j\}$. A simple calculation shows that $\ubar{\beta}(M)=2/(M-1)$ and $\ubar{\beta}^{sem}(M)=\ubar{\beta}(M-1)=2/(M-2)$. Therefore, by the previous lemma, there exists $\epsilon>0$ such that, for $2/(M-1)<\beta< 2/(M-2)+\epsilon$, $G=G^+$.$\ \qedsymbol$
\end{enumerate}

The same procedure may be applied to study the dependence of $G^+$ on $\omega=(\omega_1,...,\omega_M)$. Suppose that $\omega_i> 1,\forall i$ (this condition is not restraining at all, since any case may be reduced to this one by a simple scaling). Define
\ben
\ubar{\eta}=-\min_{1\le i\le M}{1/(\omega_i-1)}.
\een

For $\eta>\ubar{\eta}$, consider the system
\ben
\Delta u_i - (1+\eta(\omega_i-1)) u_i + \sum_{j=1}^M k_{ij}|u_j|^{p+1}|u_i|^{p-1}u_i  =0, \quad i=1,...,M.
\een
Now we write the dependence on $\eta$ as a superscript. If one defines
\ben
\mathcal{I}^\eta=\left(\inf_{J(U)=1} I^\eta(U)\right)^{\frac{p+1}{p}},
\een
and, for any $X\subset \{1,...,M\}$,
\ben
(\mathcal{I}^\eta)^X:=\left(\inf_{J(U)=1, U_i=0, i\notin X} I^\eta(U)\right)^{\frac{p+1}{p}}, \quad (\mathcal{I}^\eta)^{sem}:=\min_{X\subsetneq \{1,...,M\}} (\mathcal{I}^\eta)^X,
\een
we have once again $G=G^+$  iff $\mathcal{I}^\eta< (\mathcal{I}^\eta)^{sem}$. As before, we may show that 

\begin{prop}
For any $X\subset \{1,...,M\}$, the mapping $\eta\mapsto (\mathcal{I}^\eta)^X, \eta> \ubar{\eta}$, is continuous. In particular, $\mathcal{I}^\eta$ and $(\mathcal{I}^\eta)^{sem}$ are continuous with respect to $\eta$.
\end{prop}
%
%Suppose that $G_\beta^+\neq \emptyset$. Therefore there exists $U_\beta$ nontrivial bound-state such that
%\ben
%I(U_\beta)\le \mathcal{I}_\beta.
%\een
%Since $I(U_\beta)=J(U_\beta)$,
%\ben
%(\mathcal{I}_\beta^{sem})^p\ge \frac{I(U_\beta)^{p+1}}{J_\beta(U_\beta)}, \mbox{ i.e. } J_{NP}(U_\beta) + \beta J_P(U_\beta)\ge I(U)^{p+1}(\mathcal{I}_\beta^{sem})^{-p}.
%\een
%Hence
%\ben
%\beta\ge \frac{I(U_\beta)^{p+1}(\mathcal{I}_\beta^{sem})^{-p}-J_{NP}(U_{\beta})}{J_P(U_\beta)}=:B_\beta(U_\beta).
%\een
%Define
%\ben
%\hat{\beta} = \inf_{U\in (H^1(\real^N)\setminus\{0\})^M} B_\beta(U).
%\een
%Then $\beta<\hat{\beta}$ clearly implies $G_\beta^+=\emptyset$. Moreover, it is not hard to check that, if $\beta>\hat{\beta}$, $G_\beta^+=G_\beta$. We conclude that the function $\beta\mapsto \hat{\beta}$ has essentially all the information regarding existence or nonexistence of nontrivial ground-states.

\noindent\textit{Proof of corollary \ref{diagonalpequena}.}
First of all, notice that, if $U$ is a ground-state, $V=b^{1/2p}U$ is a ground-state of 
\begin{equation}\label{perturbado}
 \Delta u_i - \omega_i u_i + \sum_{j=1}^M \frac{k_{ij}}{b}|u_j|^{p+1}|u_i|^{p-1}u_i=0, \quad i=1,...,M.
\end{equation}
Therefore, we may consider that $b=1$ and that the diagonal terms are small. Then such a system may be seen as a $\beta$-perturbation of
\begin{equation}
 \Delta u_i -  \omega_i u_i + \sum_{j=1, i\neq j}^M |u_j|^{p+1}|u_i|^{p-1}u_i=0, \quad i=1,...,M.
\end{equation}
We note by $\mathcal{I}(\omega_1,...,\omega_M)$ the corresponding ground-state action level. By the monotonicity properties, $\mathcal{I}(\omega_1,...,\omega_M)\le \mathcal{I}(\omega_M,...,\omega_M)$. On the other hand, if $\mathcal{I}^{sem}(\omega_1,...,\omega_M)$ is the semitrivial ground-state action level (that is, the lowest action among semitrivial bound-states), then $\mathcal{I}^{sem}(\omega_1,...,\omega_M)\ge\mathcal{I}^{sem}(\omega_1,...,\omega_1)$. The proof will be concluded if one proves that
\ben
\mathcal{I}^{sem}(\omega_1,...,\omega_1)>\mathcal{I}(\omega_M,...,\omega_M).
\een
Using a suitable scaling, we have
\ben\label{scaling}
\mathcal{I}^{sem}(\omega_1,...,\omega_1)=\omega_1^{\frac{2-p(N-2)}{2p}}\mathcal{I}^{sem}(1,...,1),\quad \mathcal{I}(\omega_M,...,\omega_M)=\omega_M^{\frac{2-p(N-2)}{2p}}\mathcal{I}(1,...,1).
\een

Therefore we only have to compare the ground-state and semitrivial ground-state actions levels for 
\begin{equation}\label{nperturbado}
 \Delta u_i -  u_i + \sum_{j=1, i\neq j}^M |u_j|^{p+1}|u_i|^{p-1}u_i=0, \quad i=1,...,M.
\end{equation}

We claim that any nontrivial ground-state of \label{nperturbado} must be of the form $U=(u_i)_{1\le i\le M}$, with $u_i=(M-1)^{-\frac{1}{2p}}u_0$, where $u_0$ is a scalar ground-state: by theorem \ref{caract},
\ben
u_i=a_iu_0, \quad a_i>0
\een
Inserting this information in the system, we have
\ben\label{sistemaai}
a_i^{p-1}\sum_{j\neq i} a_j^{p+1} =1, \ i=1,...,M
\een
Suppose, w.l.o.g., that $a_1<a_2$ and $a_1\le a_i, i\ge 2$. Then
\ben
\left\{\begin{array}{l}
\sum_{j\neq 1} a_j^{p+1} = a_1^{1-p}\\
\sum_{j\neq 2} a_j^{p+1} = a_2^{1-p}
\end{array}\right.
\een
This is a contradiction, since the first sum is larger than the second. Therefore $a_1=...=a_M=a$ and so
\ben
a^{2p}(M-1)=1,
\een
yielding the claim. Now compute the action for such a ground-state:
\ben\label{Itodosiguais}
I(U)=\frac{M}{(M-1)^{\frac{1}{p}}}I(u_0).
\een
Since the mapping $M\mapsto M/(M-1)^{\frac{1}{p}}$ is strictly decreasing and any semitrivial ground-state is a nontrivial ground-state for the same system with $M-L$ equations, for some $ L\in\nat$, we have
\ben
\mathcal{I}(1,...,1)=\frac{M}{(M-1)^{\frac{1}{p}}}I(u_0)
\een
\ben
\mathcal{I}^{sem}(1,...,1)=\min_{1\le L\le M-2}\left\{\frac{M-L}{(M-L-1)^{\frac{1}{p}}}I(u_0)\right\}=\frac{M-1}{(M-2)^{\frac{1}{p}}}I(u_0).
\een
The result follows from \eqref{scaling} and hypothesis \eqref{hipoteseomegasdiferentes}. $\qedsymbol$

\vskip10pt
\noindent\textit{Proof of proposition \ref{p=3}.}
First of all, using the characterization result, any nontrivial ground-state is of the form $U=(a_iu_0)_{1\le i\le 3}$, with $u_0$ a scalar ground-state. Inserting this formula in the system and writing $b_i=a_i^2$,
\ben
\left\{\begin{array}{l}
b_1(b_2^2 + b_3^2)=1\\
b_2(b_1^2+b_3^2)=1\\
b_3(b_1^2 + b_2^2)=1
\end{array}
\right.
\een
Suppose, w.l.o.g., that $b_1\neq b_3$. Multiply the first equation by $b_1$, the third by $b_3$ and take the difference. Then 
\ben
b_2^2(b_1^2-b_3^2)=b_1-b_3, \ i.e.,\ b_2^2(b_1+b_3)=1.
\een
Define $x=b_1/b_2$, $y=b_3/b_2$. The above and the second equation imply
\ben
x^2+y^2=1/b_2^3 = x+y.
\een
Now divide the system by $b_2^3$ and take the difference between the two last equations:
\ben
x^2+y^2=y(x^2+1).
\een
Hence $x=yx^2$ and so $y=1/x$. Therefore $x^4+1=x^3+x$. One easily checks that $x=1$ is the only positive solution to this equation. Therefore $x=y$ and $b_1=b_3$, which is absurd.
Therefore $b_1=b_2=b_3$ and so $a_i=a=:2^{-1/6}$.

We observe that $V=(1,1,0)u_0$ is also a bound-state. Now compute the action of $U$ and $V$:
\ben
I(U)=\frac{3}{2^{1/6}} I(u_0) > 2I(u_0) = I(V).
\een
This means that $U$ cannot be a ground-state, which ends the proof. $\qedsymbol$
%\begin{nota}
%The result in \cite{liuwang} (despite not being explicitly shown) is also true for $p>1$. The above proof does not encompass this case because we do not know how to solve system \eqref{sistemaai}. There is however an interesting consequence: the bound-state $U^M$ such that $u^M_i=(M-1)^{-\frac{1}{2p}}Q$ is not a ground-state, for $M$ large. In fact, recalling \eqref{Itodosiguais}, we see that the action of $U^M$ increases in $M$ for $M\ge p/(p-1)+1$ (notice that this condition is verified for $M\ge 3$). Therefore \eqref{sistemaai} has more than one solution, which highlights the difficulty in solving it.
%
%\end{nota}
\end{section}
\begin{section}{System (M-NLS'): Mandel's characteristic function}

Once again, consider system \eqref{MNLSbeta}: for a given nonempty symmetric subset $P$ of $\{1,...,M\}^2$ and $\beta\in\real$,

\ben
\Delta u_i - \omega_i u_i + \sum_{(i,j)\notin P} k_{ij}|u_j|^{p+1}|u_i|^{p-1}u_i + \sum_{(i,j)\in P} \beta k_{ij}|u_j|^{p+1}|u_i|^{p-1}u_i =0, i=1,...,M
\een

For the sake of simplicity, we suppose that $k_{ij}>0, \forall (i,j)\in P$. We define
\ben
J_P(U)=\sum_{(i,j)\in P} k_{ij}\int |u_i|^{p+1}|u_j|^{p+1}, \quad J_{NP}(U)=\sum_{(i,j)\notin P} k_{ij}\int |u_i|^{p+1}|u_j|^{p+1}.
\een

As before, we shall place a subscript $\beta$ whenever a solution, function or set depends on $\beta$. Suppose that $G_\beta^+\neq \emptyset$. Therefore there exists $U_\beta$ nontrivial bound-state such that
\ben
I(U_\beta)\le \mathcal{I}_\beta.
\een
Since $I(U_\beta)=J_\beta(U_\beta)$,
\ben
(\mathcal{I}_\beta^{sem})^p\ge \frac{I(U_\beta)^{p+1}}{J_\beta(U_\beta)}, \mbox{ i.e. } J_{NP}(U_\beta) + \beta J_P(U_\beta)\ge I(U)^{p+1}(\mathcal{I}_\beta^{sem})^{-p}.
\een
Hence
\ben
\beta\ge \frac{I(U_\beta)^{p+1}(\mathcal{I}_\beta^{sem})^{-p}-J_{NP}(U_{\beta})}{J_P(U_\beta)}=:B_\beta(U_\beta).
\een
Define
\ben
\hat{\beta} = \inf_{U\in (H^1(\real^N)\setminus\{0\})^M} B_\beta(U).
\een
Then $\beta<\hat{\beta}$ clearly implies $G_\beta^+=\emptyset$. Moreover, it is not hard to check that, if $\beta>\hat{\beta}$, $G_\beta^+=G_\beta$. Also, if $\hat{\beta}=\beta$, $G_\beta\setminus G_\beta^+\neq\emptyset$.

Let us look deeper into the properties of $\hat{\beta}$. Suppose, for instance, that $\hat{\beta_0}>\beta_0$. Then
\ben
\frac{I(U)^{p+1}}{J_{\hat{\beta}_0}(U)}>(I_{\beta_0}^{sem})^p,\quad \forall U: J_P(U)\neq 0.
\een
Take $\beta_0\le \beta\le \hat{\beta}_0$. If $U_{\beta}^{sem}$ (the best semitrivial bound-state) satisfies $J_P(U_{\beta}^{sem})\neq 0$, then
\ben
(I_{\beta}^{sem})^p=\frac{I(U_{\beta}^{sem})^{p+1}}{J_{\beta}(U_{\beta}^{sem})}\ge \frac{I(U_{\beta}^{sem})^{p+1}}{J_{\hat{\beta}_0}(U_{\beta}^{sem})}>(I_{\beta_0}^{sem})^p,
\een
which is absurd, by the monotonicity properties. Therefore $J_P(U_{\beta}^{sem})=0$, for all $\beta\in [\beta_0,\hat{\beta}_0]$. In turn, by the definition of $\mathcal{I}_\beta^{sem}$, we see that it is constant in this interval and so the function $\beta\mapsto \hat{\beta}$ is constant on $[\beta_0,\hat{\beta}_0]$. Moreover, since $\beta<\hat{\beta}_0=\hat{\beta}$, $G_{\beta}^+=\emptyset$ for all $\beta\in [\beta_0,\hat{\beta})$.

So condition $\hat{\beta}>\beta$ has more implications than the simple dichotomy seen in \cite{mandel}. Precisely because of this fact, is not as powerful when studying the nonemptiness of $G^+$ as one would desire: for example, if $P$ contains all diagonal terms, one has $\beta\ge\hat{\beta}$ for all $\beta$.
\vskip10pt
$\noindent$\textit{Proof of proposition \ref{propmandel1}.}
Suppose that $U$ is a ground-state for a sequence $\beta_n\to\infty$. The hypothesis on $P$ implies that
\ben
\frac{I(U)^{p+1}}{J_\beta(U)}=\frac{I(U)^{p+1}}{J(U)}.
\een
Since, for each $\beta_n$, $U$ is the semitrivial bound-state with the lowest action, this implies that $\hat{\beta}_n=\hat{1},\forall n$. Taking $n_0$ large, $\beta_{n_0}>\hat{1}=\hat{\beta}_{n_0}$, which implies that $G=G^+$, contradicting $U\in G$.$\qedsymbol$
\vskip10pt
$\noindent$\textit{Proof of proposition \ref{propmandel2}}.
Through a normalization, one may assume $\mu=1$. From \cite{mandel}, the property is true for $M=2$. We now proceed by induction: suppose that the result is true for $M-1$ equations. Then there exists $\beta_{M-1}$ and $U_0$ with only one nonzero component such that 
\ben\label{447}
\frac{I(U)^{p+1}}{J_{\beta}(U)}\ge \frac{I(U_0)^{p+1}}{J_{\beta}(U_0)}=\frac{I(U_0)^{p+1}}{J(U_0)}=I(U_0)^p, \forall U\mbox{ semitrivial }, \forall 0<\beta<\beta_{M-1}.
\een
Consider the function $\beta\mapsto\hat{\beta}$. Since $U_0$ has only one nonzero component, $\hat{\beta}$ is constant on $(0,\beta_{M-1})$. Take any $U$ such that $J_P(U)\neq 0$. W.l.o.g., assume that the last component has the largest $L^{2p+2}$ norm. For each $1\le i\le M$, define

\ben
r_i=\frac{\|u_i\|_{2p+2}}{\|u_M\|_{2p+2}}\le 1,\ V_i=((v_i)_1,...,(v_i)_M),\ (v_i)_j=u_i\delta_{ij}
\een
Then, using \eqref{447} and $J_\beta(V_i)=J(V_i)$,
\begin{align*}
&\frac{I(U)^{p+1}(I(U_0))^{-p}-J_{NP}(U)}{J_P(U)}=
\frac{\left(\sum_{i=1}^M I(V_i)\right)^{p+1}I^{-p}(U_0) - \sum_{i=1}^M J(V_i)}{J_P(U)}\\
=\ &\frac{\left(\frac{I(V_M)}{J(V_M)^{1/(p+1)}} + \sum_{i=1}^{M-1} r_i^2 \frac{I(V_i)}{J(V_i)^{1/(p+1)}}\right)^{p+1}I^{-p}(U_0) - 1-\sum_{i=1}^{M-1}r_i^{2p+2}}{\frac{J_P(U)}{J(V_M)}}
\\ \ge\ & \frac{(1+\sum_{i=1}^{M-1}r_i^2)^{p+1}-1-\sum_{i=1}^{M-1}r_i^{2p+2}}{2\sum_{i=1}^{M-1}r_i^{p+1} + \sum_{i,j=1}^{M-1}r_i^{p+1}r_j^{p+1}}=:g(r_1,...,r_{M-1})
\end{align*}
Since $p\ge 1$, $g$ is bounded below over the set $[0,1]^{M-1}$ by a constant $m>0$. Hence $B_\beta(U)\ge m>0, \forall U$. Therefore, taking $\beta_M=\min\{\beta_{M-1}, m\}$, we see that $\hat{\beta}\ge m>\beta$, for $0<\beta<\beta_M$. The properties of $\hat{\beta}$ imply that the result is true for $M$ equations. $\qedsymbol$

\end{section}

\begin{section}{Examples}

\begin{exemplo}
Consider $M=3$, $p=1$ and suppose that the coefficient matrix $K$ is of the form
\ben
K=\left[\begin{array}{ccc}
0 & a & b\\
a & 0 & c\\
b & c & 0
\end{array}\right],\quad a\le b\le c\in\real\setminus\{0\}.
\een
Now one must divide in several cases:
\bi
\item $c<0$: in this case, the condition for the existence of ground-states is not verified and so there are no ground-states;
\item $c>0, b<0$: applying theorem 6 of \cite{simao}, any ground-state satisfies either $u_1=0$ or $u_2,u_3=0$. The second possibility implies that $u_1$ satisfies $-\Delta u_1 + u_1=0$, which is impossible. Therefore $u_1=0$. Since $U=(a_iu_0)_{1\le i\le 3}$, a direct substitution on the system gives $a_2=a_3=c^{-1/2}$.
\item $b>0$: suppose that $U$ is a nontrivial ground-state. Then, inserting the characterization formula in the system (M-NLS), we obtain
\ben\label{sistemaX}
KX=(1,1,1)^T,\quad X=(a_1^2,a_2^2,a_3^2)^T.
\een
The determinant of $K$ is $2abc$. Now, using Cramer's rule,
\ben
a_1^2= \frac{a+b-c}{2ab}, \ a_2^2=\frac{a+c-b}{2ac},\ a_3^2=\frac{b+c-a}{2bc}.
\een
This implies that $(a+b-c)a>0$ and $(a+c-b)a>0$. Now, if $V$ is a semitrivial ground-state, using the characterization and the fact that $c\ge a,b$,
\ben
V=\left(0,\frac{1}{\sqrt{c}},\frac{1}{\sqrt{c}}\right)u_0.
\een
Now, comparing the actions of these two solutions, the condition for the existence of nontrivial ground-states is 
\ben
\frac{-a^2-b^2-c^2 + 2ab + 2bc + 2ac}{2ab}\le 2
\een
which, for $a>0$, simplifies to
\ben\label{314}
2c(a+b)\le (a+b)^2+c^2, i.e., (a+b-c)^2\ge 0.
\een
Therefore, for $a>0$, one has the following:
\bi
\item if $a+b\le c$, $G^+=\emptyset$, since system \eqref{sistemaX} has no positive solutions. This implies
\ben
G=\left\{\left(0,\frac{1}{\sqrt{c}},\frac{1}{\sqrt{c}}\right)u_0: u_0\in G_{(1-NLS)}\right\}.
\een

\item if $a+b>c$, then
\ben
G=\left\{\left(\sqrt{\frac{a+b-c}{2ab}}, \sqrt{\frac{a+c-b}{2ac}}, \sqrt{\frac{b+c-a}{2bc}}\right)u_0: u_0\in G_{(1-NLS)}\right\}.
\een
\ei

For $a<0$, inequality \eqref{314} is reversed and strict, hence $G^+=\emptyset$. 
\ei
Hence the necessary and sufficient condition for the existence of nontrivial ground-states with $a\le b\le c\in\real$ is $a+b>c$.
\end{exemplo}

\begin{exemplo}
Consider $M=3$, $p=1$ and suppose that the coefficient matrix $K$ is of the form
\ben
K=\left[\begin{array}{ccc}
1 & a & b\\
a & 1 & c\\
b & c & 1
\end{array}\right],\quad 1\ll a\le b\le c
\een
The previous example may be seen as a limit when $a,b,c$ are very large. With some computations, one derives the following:
\bi
\item The possible semitrivial ground-state is given by
\ben
V=\left(0,\frac{1}{\sqrt{1+c}},\frac{1}{\sqrt{1+c}}\right).
\een
\item The possible nontrivial ground-state, $U=(a_i u_0)_{1\le i\le 3}$, is given by
$$
a_1^2=\frac{1+(a+b)c - a-b-c^2}{1+2abc-a^2-b^2-c^2},\ a_2^2=\frac{1+(a+c)b - a-c-b^2}{1+2abc-a^2-b^2-c^2}, 
$$
\ben
\ a_3^2=\frac{1+(c+b)a - c-b-a^2}{1+2abc-a^2-b^2-c^2}.
\een
We assume that $a,b,c$ are such that all numerators and denominators above are positive. Notice that this is true for $a,b,c$ large enough and $a+b>c$.
\ei

As in the previous example, if one compares the corresponding action levels, one has $G^+\neq\emptyset$ iff
\ben
0\le (a+b-c)(a+b-c-2).
\een
Since we assumed that $a+b>c$, the condition is simply $a+b\ge c+2$. We see that, even for systems where the couplings $k_{ij}, i\neq j$, are large comparing to the diagonal terms $k_{ii}$, one may have $G^+=\emptyset$. This does not go against the conclusion of corollary \ref{diagonalpequena} and the perturbation arguments: the problem here is that $a, b$ and $c$ are not close to each other. This example shows that, in order for one to have $G^+\neq\emptyset$, one must take into account the relation between coupling coefficients.
\end{exemplo}

\begin{exemplo}
Consider system (3-NLS), $p=1$ and the coefficient matrix
\ben
K=\left[\begin{array}{ccc}
0 & b & 1\\
b & 0 & 2\\
1 & 2 & \mu
\end{array}\right],\quad  b>0, \mu\in\real.
\een
Using the characterization, everything is reduced to the study of the proportionality constants $a_1, a_2$ and $a_3$. For the sake of simplicity, $x=a_1^2$, $y=a_2^2$, $z=a_3^2$. It is now a simple calculation to obtain the following:
\bi
\item Semitrivial A: The possible ground-state with $x=0$ satisfies $y=(2-\mu)/4$, $z=\frac{1}{2}$. This solution only exists if $2>\mu$.
\item Semitrivial B: Analogously, the possible ground-state with $x=0$ satisfies $x=1-\mu$, $z=1$. This solution only exists if $1>\mu$.
\item Semitrivial C: The possible ground-state with $z=0$ satisfies $x=y=1/b$;
\item Semitrivial D: If $x=y=0$, then $z=1/\mu$. This solution only exists if $\mu>0$;
\item Nontrivial E: For the possible nontrivial ground-state,
\ben
x=\frac{\mu b - 2b +2}{b(\mu b-4)}, y=\frac{\mu b - b -1}{b(\mu b-4)},\quad z=\frac{b-3}{\mu b-4}.
\een
This solution exists only if $b>3$ and $\mu>2(b-1)/b$ or if $b<3$ and $\mu<2(b-1)/b$.
\ei
The action for each of these solutions is (up to a constant)
\ben
A: \frac{4-\mu}{4},\quad B:2-\mu,\quad C=2/b,\quad D: 1/\mu \quad  E:\frac{b^2+(2\mu-6)b+1}{b(\mu b-4)}
\een

Now we compare the various actions, whenever the solutions exist:
\begin{enumerate}
\item First of all, $A$ is always lower than $B$ so one may discard this solution;
\item The solution $D$ is best if $\mu>2, b/2$;
\item A tideous computation shows that $E$ is the ground-state if $b<3$ and $\mu<2(b-1)/b$;
\item In the remaining area, $A$ is better than $C$ if $2>\mu>4-8/b$.
\end{enumerate}

Intersecting these comparisons with the domains where each solution exists, we obtain diagram 2, which is already revealing of the complexity of this problem.
\begin{figure}[h]\label{fig2}
\centering
\includegraphics[width=10cm, height=7cm]{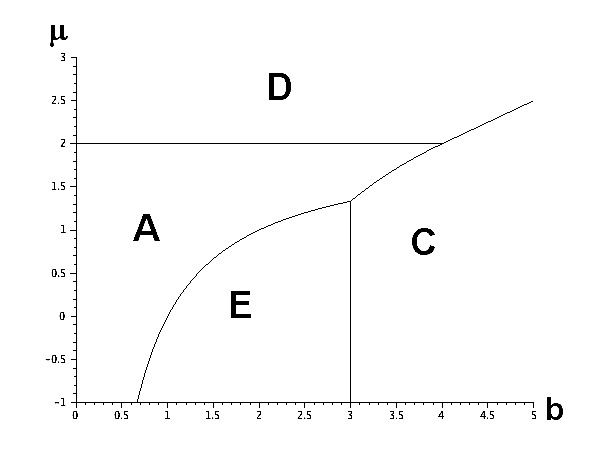}
\caption{Regions of the $b-\mu$ plane where each solution is a ground-state.}
\end{figure}

Several remarks are necessary:
\begin{enumerate}
\item First of all, we see that, for $b<3$, when $-\mu$ is very large, the ground-state is nontrivial. Moreover, if $b>3$, no value of $\mu$ produces nontrivial ground-states;
\item One might think that some solutions (for example, the nontrivial one), if they exist, would always have minimal action. However, the reader may check that this is not true for this system;
\item It is natural that one never has $u_1\neq 0$ and $u_2=0$, since the coupling coefficients associated with the second component are larger than the coefficients associated with the first component. In fact, this is a consequence of the monotonicity with respect to the coefficient matrix;
\item Take, for example, $b=4$. One sees that the mapping $\mu\mapsto \hat{\mu}$ is constant up to $\mu=2$ and $\hat{\mu}=\mu$ for $\mu>2$. On the other hand, if one sets $\mu=1$, one observes that mapping $b\mapsto \hat{b}$ is constant up to $b=2$, $b>\hat{b}$ for $2<b<3$ and $\hat{b}=b$ for $b>3$. In any case, one may observe that these mappings are continuous.
\end{enumerate}

\end{exemplo}
\end{section}
\begin{section}{Further comments}

\indent One of the main ideas that it should be clear at the end of this work is that the system (M-NLS') for $M=2$ has a much simpler structure than the case $M=3$. The examples we presented put in evidence the complexity of this problem. Using the characterization theorem, one should build more examples to see which properties one may expect or not. It would be especially interesting to build a nontrivial example for $p>1$ and see how does the set $G^+$ evolve as a function of the parameters.

Another problem related with system (M-NLS') is the existence of bound-states with the lowest action among \textit{nontrivial} bound-states. This is not trivial at all, especially because it lacks a suitable variational formulation. Some attempts, using generalized Nehari manifolds, have proven the existence of such bound-states. It would be interesting to see if one may extend the characterization theorem to this case.

One of the reasons for which ground-states are an interesting object to study is because they give rise to periodic solutions for (1.1). In this context, one may study the stability of these solutions. It is known (see \cite{cazenave}, \cite{simao2}, \cite{maiamontefuscopellacci2}) that the variational properties of the ground-states influence deeply their stability. We would like to point out the following: using the characterization theorem, we see that there exists a bijection between the set of ground-states and the set of solutions of a constrained maximization problem over $\real^M$. Now consider \textit{local} solutions of the same constrained problem in $\real^M$. These solutions give rise to bound-states, which may or may not be ground-states. However, the local maximization property should be enough to prove results on local stability.

%
%The same procedure may be applied to study the dependence of $G^+$ on $\omega$. Suppose that $\omega\ge 1$ (this condition is not restraining at all, since any case may be reduced to this one by a simple scaling). For $\eta>0$, consider the system
%\ben
%\Delta u_i - (1+\eta(\omega_i-1)) u_i + \sum_{j=1}^M k_{ij}|u_j|^{p+1}|u_i|^{p-1}u_i  =0
%\een
%If one defines
%\ben
%\hat{\eta} = \sup_{U\in (H^1(\real^N)\setminus\{0\})^M} \frac{J(U)^{\frac{1}{p+1}}(\mathcal{I}_\beta^{sem})^{\frac{p}{p+1}}-\left(\sum_{i=1}^M \|u_i\|_{H^1}^2\right)}{\sum_{i=1}^M (\omega_i-1)\|u_i\|_2^2},
%\een
%we obtain
%\begin{enumerate}
%\item $\hat{\eta}>\eta \Rightarrow G^+_\eta = G_\eta$;
%\item $\hat{\eta}<\eta \Rightarrow G^+_\eta =\emptyset$.
%\end{enumerate}
%
%Therefore, we shall concentrate our study on the properties of functions $\beta\mapsto \hat{\beta}$ and $\eta\mapsto\hat{\eta}$. First of all, considering the lemmata \ref{mon1}, \ref{mon2}, we see that $\hat{\beta}$ and $\hat{\eta}$ are increasing functions.

\end{section}
\begin{section}{Acknowledgements}
The author was partially supported by Fundação para a Ciência e Tecnologia, through the grants SFRH/BD/96399/2013 and UID/MAT/04561/2013. The author would like to thank Hugo Tavares for some important suggestions.
\end{section}

\end{document}